\documentclass[11pt]{amsart}
\usepackage{amssymb}
\usepackage[all]{xy}

\newcommand{\dis}{\displaystyle}

\def\N{{\mathbb N}}
\def\B{{\mathcal B}}

\def\m{{\mathfrak m}}

\newtheorem{theorem}{Theorem}[section]
\newtheorem{lemma}[theorem]{Lemma}

\theoremstyle{definition}
\newtheorem{definition}[theorem]{Definition}

\newtheorem{proposition}[theorem]{Proposition}

\theoremstyle{remark}
\newtheorem{remark}[theorem]{Remark}

\numberwithin{equation}{section}

\begin{document}

\title{Liftings of a monomial curve}

\author{Mesut \c{S}ah\. in}
\address{Department of Mathematics,
Hacettepe University, Beytepe,  06800, Ankara, Turkey}
\email{mesut.sahin@hacettepe.edu.tr}
\date{\today}


\commby{}

\dedicatory{}

\begin{abstract}
We study an operation, that we call lifting, creating non-isomorphic monomial curves from a single monomial curve. Our main result says that all but finitely many liftings of a monomial curve have Cohen-Macaulay tangent cones even if the tangent cone of the original curve is not Cohen-Macaulay. This implies that the Betti sequence of the tangent cone is eventually constant under this operation. Moreover, all liftings have Cohen-Macaulay tangent cones when the original monomial curve has a Cohen-Macaulay tangent cone. In this case, all the Betti sequences are nothing but the Betti sequence of the original curve.
\end{abstract}

\maketitle
\section{introduction} The main theme of this note is an operation that we call lifting and use for spreading a special property of a monomial curve within a family of infinite examples. By a monomial curve $C_S$ in the affine space $K^n$ over a field $K$, we mean a curve whose points $(t^{m_{1}},\dots, t^{m_{n}})$ varies with the parameter $t\in K$. The curve has a strong relation with the semigroup $S$ generated minimally by the positive integers $m_1,\dots,m_n$. The smallest minimal generator of $S$ is called the \textit{multiplicity} of $C_S$. The monomials $t^{m_{1}},\dots, t^{m_{n}}$ generate the integral domain $K[S]$ of $K[t]$, which is known as the \textit{numerical semigroup ring} associated to $S$. The polynomial ring $R=K[x_1,\dots,x_n]$ is graded by the semigroup $S$, by setting $\deg _S (x_i)=m_i$, giving rise to a graded map $R \rightarrow K[S]$, which sends $x_i$ to $t^{m_i}$. Its kernel denoted $I_S$ is called the \textit{toric ideal} of $S$. When $K$ is algebraically closed, the vanishing ideal $I(C_S)$ of $C_S$ is $I_S$, and thus $K[S]$ is isomorphic to the coordinate ring $R/I_S$ of the curve $C_S$. The curve $C_S$ is singular at the origin if $S\neq \N$ in which case one wants to measure how nice the singularity is. An algebraic way to understand the quality of the singularity is to look at the coordinate ring $gr_{\mathfrak{m}}(K[S])$ of the \textit{tangent cone} which is isomorphic to the ring $R/I^*_{S}$, where $\m=\langle t^{m_{1}},\dots, t^{m_{n}} \rangle$ is the maximal ideal of $K[S]$. Here $I^*_{S}$ is the ideal generated by the least degree homogeneous summands $f^{*}$ of $f$ in $I_S$. Many researchers investigated Cohen-Macaulayness of $gr_{\mathfrak{m}}(K[S])$ for this purpose, see e.g. \cite{arslan,CZ,huang,shen,shibuta}. 

It is well known that the ideal $I_S$ is generated by finitely many binomials $x_1^{u_1}\cdots x_n^{u_n}-x_1^{v_1}\cdots x_n^{v_n}$ such that $u_iv_i=0$ and $\deg_S(x_1^{u_1}\cdots x_n^{u_n})=\deg_S(x_1^{v_1}\cdots x_n^{v_n})$. In general, the ideal $I_S$ has alternative minimal generating sets of the same cardinality. If a binomial $B$ or $-B$ appears in every minimal generating set it is called \textit{indispensable}. Hence, the ideal $I_S$ has a unique minimal generating set if and only if it is generated by indispensable binomials. Uniquely generated ideals are of interest in the emerging field of algebraic statistics, see \cite{goj,tak}. Recently, indispensable binomials are also used for characterizing monomial curves in $K^4$ with Cohen-Macaulay tangent cone, see \cite{AKN,pseudo}. 

Our main aim in this note is to see how Cohen-Macaulayness spreads among liftings of a monomial curve. It is inspired by a recent paper by Herzog and Stamate (see \cite{HS}) on a similar additive operation on $S$ called shifting, which produces a semigroup generated by $m_1+k,\dots,m_n+k$ for every positive integer $k$. In that article, Herzog and Stamate show that tangent cones of shiftings are Cohen-Macaulay for all sufficiently large $k$, and that the Betti numbers of the tangent cones are eventually periodic in $k$.  The second concern of this paper is to examine how lifting effects indispensability of binomials in $I_S$ and strong indispensability of a minimal free resolution of $K[S]$. Before stating our main results let us introduce more notations. Fix a numerical semigroup $S$ generated minimally by $m_1 ,m_2,\dots ,m_n$. By a $k$-lifting $S_{k}$ of $S$ we mean the numerical semigroup generated by $$m_1,km_{2},\dots,km_n,$$ where $k$ is a positive integer with $\gcd (k,m_1)=1$. The last condition is to avoid having isomorphic liftings. In the same vein, the monomial curve $C_{k}$ corresponding to $S_{k}$ is called a $k$-lifting of $C:=C_S$. When $C$ has multiplicity $m_1$, all of its liftings will have multiplicity $m_1$. There is a closely related operation called simple gluing that have been used in literature to produce more examples with an interesting structure from a single monomial curve. Let $T$ be a semigroup generated minimally by $m_{2},\dots,m_n$. If $m_1\in T$ and $\gcd(k,m_1)=1$, then $S_k$ is called a simple gluing of $T$ and $\N$. Notice that $S_k$ is a $k$-lifting of $S_1=S$ for any $k$.  This technique has been used for the first time in \cite{watanabe} to prove that $S_k$ is symmetric if and only if $T$ is symmetric. It is used in \cite{morales} to produce monomial curves with Noetherian symbolic blow-ups and in \cite{thoma} to create monomial curves which are set theoretic complete intersections. In \cite{pf}, the authors proved that the tangent cone $gr_{\mathfrak{m}}(K[S_k])$ is Cohen-Macaulay when the same is true for the semigroup $T$ and $k\leq m_2+\cdots+m_n$. This is extended later in \cite{JZ}. Finally, it has very recently been used to study catenary degree (see \cite[Theorem 3.3]{Oneill}) which is an invariant of the semigroup measuring complexity of factorisations of elements. As $T$ is independent of $k$, all these results reveal common behavior of the liftings $S_k$, in the special case that $m_1\in T$. 

We finish the introduction by describing the structure of the paper. In the next section, we establish a one to one correspondence between the binomials in both ideals $I_S$ and $I_{S_k}$ with monomials having no common divisor, associating minimal generators of $I_S$ to those of $I_{S_k}$, and preserving indispensability, see Proposition \ref{indispensables}. In Section $3$, we show that all but finitely many liftings of a monomial curve have Cohen-Macaulay tangent cones even if the tangent cone of the original curve is not Cohen-Macaulay. Moreover, all liftings have Cohen-Macaulay tangent cones when the original monomial curve has a Cohen-Macaulay tangent cone, see Theorem \ref{liftsofCM}. In the last section, we show that the minimal free resolution of $K[S]$ is strongly indispensable if and only if the minimal free resolution of $K[S_k]$ is strongly indispensable, see Proposition \ref{free}. Finally, we prove that the Betti sequence of the tangent cone is eventually constant under this operation and that all the Betti sequences are nothing but the Betti sequence of the original tangent cone if the latter is Cohen-Macaulay, see Theorem \ref{resTC}.

\section{Indispensability}

In this section we establish the correspondence between the indispensable binomials of a monomial curve and those of its liftings. Let us recall a graph encoding minimal generators of $I_S$. Let $V( d)$ be the set of monomials of $S$-degree $d$. Denote by $G(d)$ the graph with vertices the elements of $V( d)$ and edges $\{M,N\} \subset V( d)$ such that the binomial $M-N$ is generated by binomials in $I_S$ of $S$-degree strictly smaller than $d$ with respect to $<_S$, where $s_1<_S s_2$ if $s_2-s_1 \in S$. In particular, when $\gcd(M,N) \neq 1$, $\{M,N\}$ is an edge of $G( d)$ as $M-N=\gcd(M,N)\cdot [(M-N)/\gcd(M,N)]$ and the binomial $(M-N)/\gcd(M,N)$ in $I_S$ has strictly smaller $S$-degree. A binomial of $S$-degree $d$ is indispensable if and only if $G(d)$ has two connected components which are singletons, by \cite[Corollary 2.10]{CKT}. This means that $M-N$ is indispensable exactly when 
$V( d)=\{M,N\}$ and $M-N$ is not generated by binomials in $I_S$ of $S$-degree strictly smaller than $d$.

Let $b \in S$. Every tuple $(v_1,\dots,v_n)\in \N^n$ satisfying $b=v_1 m_1+\cdots+v_n m_n$ is called a \textit{factorization} of $b$. Note that there is a one to one correspondence between factorizations $(v_1,\dots,v_n)$ of $b\in S$ and monomials $N=x_1^{v_1}\cdots x_n^{v_n}$ of $S$-degree $b$. 
The following key fact will be used many times in the sequel.
\begin{lemma}\label{semigroups}  There is a one to one correspondence between $S$ and $kS\subset S_k$ under which $m_1+S$ is mapped onto $k(m_1+S)$. Moreover, the factorization $(v_1,v_2,\dots,v_n)$ of $b\in S$ corresponds to the factorization $(kv_1,v_2,\dots,v_n)$ of $kb\in kS$.  
\end{lemma}
\begin{proof} If $b \in S$, then $k b \in k S \subset S_k$. Conversely, if $k b \in S_k$, then we have $k b=v'_1 (m_1)+v_2 (k m_2)+\cdots+v_n (k m_n)$. So, $k$ divides $v'_{1}m_1$, which forces the existence of $v_{1}\in \N$ such that $v'_{1}=k v_{1}$, since $\gcd(k,m_1)=1$. Thus,
$b=v_1 m_1+\cdots+v_n m_n \in S$. So, the map sending $b\in S$ to $kb=kv_1 (m_1)+v_2 (km_2)+\cdots+v_n (km_n) \in S_k$ is one to one and onto. As $v_1\neq 0$ if and only if $kv_1\neq 0$, it follows that this map restricts to a one to one correspondence between $m_1+S$ and $k(m_1+S)$. Clearly, the factorization $(v_1,v_2,\dots,v_n)$ of $b$ corresponds to the factorization $(kv_1,v_2,\dots,v_n)$ of $kb$.
\end{proof}

It is time to describe the correspondence between the binomials in both ideals $I_S$ and $I_{S_k}$ with monomials having no common divisor. This has been noted for the first time by Morales (see \cite[Lemma 3.2]{morales}). We show that indispensable binomials are associated with indispensable ones under this correspondence. Let $B$ be a binomial in $I_S$ with monomials having no common divisor. Then, either $x_1$ divides no monomials in $B$ or it divides only one of them. In the former case, $B$ lies in $I_{S_k}$, so let $B_k:=B$. In the latter, $B=M-N$ with $M=x_2^{u_2}\cdots x_n^{u_n}$ and $N=x_1^{v_1}x_2^{v_2}\cdots x_n^{v_n}$, so we let $B_k:=M-N_k \in I_{S_k}$, where $N_k=x_1^{kv_1}x_2^{v_2}\cdots x_n^{v_n}$. In both cases, $\deg_{S_k}(B_k)=k\deg_S(B)\in kS$.
\begin{proposition} \label{indispensables} The map $\phi_k: I_S \rightarrow I_{S_k}$,  given by $B \rightarrow B_k$, is a one to one correspondence between the binomials in both ideals with monomials having no common divisor, associating minimal generators of $I_S$ to those of $I_{S_k}$, and preserving indispensability.
\end{proposition}
\begin{proof}  It is clear that the assignment is one to one. We prove that it is onto by using Lemma \ref{semigroups}. Now, assume that $M-N'$ is a binomial in $I_{S_k}$, where $M=x_2^{u_2}\cdots x_n^{u_n}$ and $N'$ is a monomial corresponding to the factorization $(v'_1,v_2,\dots,v_n)$ with $v'_1>0$. So, $\deg_{S_k}(N')=\deg_{S_k}(M)=kb$ lies in $k(m_1+S)$, where $b:=\deg_{S}(M)$. Lemma \ref{semigroups} implies that the factorization $(v'_1,v_2,\dots,v_n)$ of $kb$ corresponds to the factorization $(v_1,v_2,\dots,v_n)$ of $b$, for some $v_1>0$. Therefore, there is a monomial $N=x_1^{v_1}x_2^{v_2}\cdots x_n^{v_n}$ with $S$-degree $b$ such that $N'=N_k$. So, $\phi_k(M-N)=M-N'$.

Let $G(b)$ be the graph of an $S$-degree $b$ and $G_k(d)$ be the graph of an $S_k$-degree $d$. By Lemma \ref{semigroups}, there is a one to one correspondence between the monomials of $S$-degree $b$ and monomials of $S_k$-degree $kb$. So, there is a correspondence between the vertices of the graphs $G(b)$ and 
$G_k(kb)$. By the first part, $M-N$ is generated by binomials in $I_{S}$ of $S$-degree smaller than $b$ if and only if $M-N_k$ is generated by binomials in $I_{S_k}$ of $S_k$-degree smaller than $kb$. This correspondence associates edges between $\{M\}$ and $\{N\}$ to edges between $\{M\}$ and $\{N_k\}$. As these graphs determine the minimal generators, this correspondence associates a minimal generating set of $I_S$ to a minimal generating set of $I_{S_k}$. Assuming that $B$ is an indispensable binomial of $S$-degree $b$, we get that $G(b)$ has two connected components $\{M\}$ and $\{N\}$ which are singletons by \cite[Corollary 2.10]{CKT}. Thus, $\{M\}$ and $\{N_k\}$ are the only connected components of $G_k(kb)$ and they are singletons. Hence, $B$ is indispensable if and only if $B_k$ is indispensable.
\end{proof}

\section{Cohen-Macaulayness of the tangent cone}

In this section, we study local properties of liftings of a monomial curve $C$ with multiplicity $m_1$. Recall that the $S$-degree and the usual degree of a monomial is 
$$\deg_S (x_1^{u_1}\cdots x_n^{u_n})= u_1m_1+\cdots+u_n m_n; \quad \deg (x_1^{u_1}\cdots x_n^{u_n})= u_1+\cdots+u_n.$$

\begin{lemma}[\cite{herzog}] \label{Herzog} Let $C$ be the monomial curve corresponding to $S$. Then the tangent cone of $C$ is Cohen-Macaulay if and only if for every monomial $M=x_2^{u_2}\cdots x_n^{u_n}$ with $\deg_S(M) \in m_1+S$ there exists a monomial $N=x_1^{v_1}\cdots x_n^{v_n}$ with $v_1>0$ such that 
$\deg_S M = \deg_S N$ and $ \deg M \leq \deg N. \hfill \Box$

\end{lemma}

\begin{remark} \label{remark1} It is sufficient to check the conditions in Lemma \ref{Herzog} for monomials $M$ with $u_i< m_1$, where $i=2,\dots,n$. This is because when $u_i \geq m_1$ we have $u_i=q_im_1+r_i$ and so $$\deg_S(M)=\sum_{j=2}^{n} u_jm_j=(q_im_i)m_1+r_im_i+\sum_{j\in\{2,\dots,n\}\setminus \{i\}} u_jm_j=\deg_S(N)$$ 
for $N=x_1^{q_im_i}\cdots x_i^{r_i} \cdots x_n^{u_n}$. Since $u_i=q_im_1+r_i<q_im_i+r_i$, it follows that $$\deg (M)=\sum_{j=2}^{n} u_j<(q_im_i)+r_i+\sum_{j\in\{2,\dots,n\}\setminus \{i\}} u_j=\deg(N).$$ 
\end{remark}

\begin{theorem}\label{liftsofCM} Let $C$ be a monomial curve and $C_k$ be its $k$-th lift. If $C$ has a Cohen-Macaulay tangent cone, then $C_k$ has a Cohen-Macaulay tangent cone for all $k>1$. If not, there is a positive integer $k_0$ such that $C_k$ has a Cohen-Macaulay tangent cone for all $k\geq k_0$.
\end{theorem}
\begin{proof} Take a monomial $M=x_2^{u_2}\cdots x_n^{u_n}$  whose $S_k$-degree $\deg_{S_k}(M)=k\deg_{S}(M)$ lies in $m_1+S_k$. It follows from Lemma \ref{semigroups} that $\deg_{S}(M)\in m_1+S$. This means that there is a monomial $N=x_1^{v_1}x_2^{v_2}\cdots x_n^{v_n}$ such that $v_1>0$ and $\deg_{S}(M)=\deg_{S}(N)$. Clearly, we also have $\deg_{S_k}(M)=\deg_{S_k}(N_k)$, for $N_k=x_1^{kv_1}x_2^{v_2}\cdots x_n^{v_n}$. So, under the correspondence in Proposition \ref{indispensables}, the binomial $M-N$ in $I_S$ maps to $M-N_k$ in $I_{S_k}$.

On the other hand, we have
\begin{eqnarray} \label{eqn3}
\deg(N_k)-\deg(M)&=&\deg(N_k)-\deg(N)+\deg(N)-\deg(M) \\
\nonumber &=&(k-1)v_1+\deg(N)-\deg(M).
\end{eqnarray}

Now, if $C$ has a Cohen-Macaulay tangent cone, it follows from Lemma \ref{Herzog} that there is at least one $N$ among all with 
$\deg(N)-\deg(M)\geq 0$. Thus, we have $\deg(N_k)-\deg(M)\geq 0$ for the corresponding $N_k$ by Equation \ref{eqn3}. So, $C_k$ has a Cohen-Macaulay tangent cone by Lemma \ref{Herzog}. 

If $C$ does not have a Cohen-Macaulay tangent cone, then again by Lemma \ref{Herzog} there is some monomial $M$ with $\deg_{S}(M)\in m_1+S$ such that for all $N$ with $\deg_{S}(M)=\deg_{S}(N)$ we have $\deg(N)-\deg(M)< 0$. Since $S$ is a numerical semigroup, there are only finitely many monomials $N$ with the $S$-degree $\deg_{S}(M)$ for a fixed monomial $M$. Let $N_0$ be the one with the biggest degree so that $\deg(N)-\deg(M)$ is the biggest possible negative number. Then, for large enough $k_0(M)$, we have $(k_0(M)-1)v_1+\deg(N)-\deg(M)\geq 0$. Hence, it follows that $\deg(N_k)-\deg(M)\geq 0$ by Equation \ref{eqn3}, for all $k\geq k_0(M)$. By Remark \ref{remark1} it is sufficient to check the condition of Lemma \ref{Herzog} for finitely many monomials $M$. So, if we choose $k_0$ to be the maximum of all $k_0(M)$ corresponding to these monomials, then Lemma \ref{Herzog} completes the proof.
\end{proof}

\section{Minimal free resolutions}

Let $S$ be a numerical semigroup and $S_k$ be its $k$-th lift as before. In this section we discuss the relation between their homological invariants.  We start with the relation between minimal free resolutions of the semigroup rings $K[S]$ and $K[S_k]$. Recall that $R=K[x_1,\dots,x_n]$ is $S$- and $S_k$-graded, respectively, via 
$$\deg_S(x_1)=\deg_{S_k}(x_1)=m_1 \quad \mbox{and} \quad \deg_{S_k}(x_i)=k\deg_S(x_i)=k m_i, \quad i>1.$$ As indicated in \cite{numata}, a minimal $S_k$-graded free resolution of $K[S_k]$ is obtained from a minimal $S$-graded free resolution of $K[S]$ via the faithfully flat extension $f:R \rightarrow R$, defined by sending $x_1 \rightarrow x_1^{k}$ and  $x_i \rightarrow x_i$ for all $i>1$. These prove the following

\begin{proposition}\label{free} If $K[S]$ has a minimal $S$-graded free resolution given by
$$\dis 0  {\longrightarrow} \bigoplus_{j=1}^{\beta_{n-1}} R[{-b_{n-1,j}}]  {\longrightarrow} \cdots     {\longrightarrow} \bigoplus_{j=1}^{\beta_1} R[{-b_{1,j}}]   {\longrightarrow} R {\longrightarrow} K[S]{\longrightarrow}0,$$
then  $K[S_k]$ has a minimal $S_k$-graded free resolution given by
$$\dis  0  {\longrightarrow} \bigoplus_{j=1}^{\beta_{n-1}} R[{-k b_{n-1,j}}]  {\longrightarrow} \cdots     {\longrightarrow} \bigoplus_{j=1}^{\beta_1} R[{-k b_{1,j}}]   {\longrightarrow} R {\longrightarrow} K[S]{\longrightarrow}0.   $$  \hfill \mbox{ $\Box$} 
\end{proposition}

It follows from Proposition \ref{indispensables} that $I_{S_k}$ has a unique minimal generating set or equivalently is generated minimally by indispensable binomials if and only if $I_S$ has the same property. This means that one of the first matrices in the resolutions of Proposition \ref{free} is unique if and only if the other is so. The corresponding notion introduced by Charalambous and Thoma (see \cite{haraJA}) for the full minimal free resolution is strong indispensability of the resolution. Recall that a resolution $({\bf F},\phi)$ is called strongly indispensable if for any graded minimal resolution $({\bf G},\theta)$, we have an injective complex map
$i\colon({\bf F},\phi)\longrightarrow({\bf G},\theta)$. As a consequence, we get the following result about indispensability of higher syzygies using Lemma \ref{semigroups}.
\begin{proposition} \label{strongly}$K[S_k]$ has a strongly indispensable minimal free resolution $\iff$ $K[S]$ has a strongly indispensable minimal free resolution.
\end{proposition}
\begin{proof}  By \cite[Lemma 19]{bafrsa}, the algebra $K[S_k]$ has a strongly indispensable minimal free resolution $\iff$ $b-b' \notin S_k$, for all $b,b' \in \B_i(S_k)$. Proposition \ref{free} implies that $\B_i(S_k)=k \B_i(S)$ and thus $b \in \B_i(S_k) \iff b=k d$, for some $d \in \B_i(S)$. Thus, $b-b' \notin S_k$ for all $b,b' \in \B_i(S_k) \iff d-d' \notin S$ for all $d,d' \in \B_i(S)$, by Lemma \ref{semigroups}, which completes the proof by \cite[Lemma 19]{bafrsa} again.
\end{proof}

\begin{definition} For an ideal $I$, a finite subset $G \subset I$ is called a standard basis of $I$ if the least homogeneous summands of the elements of $G$ generate the $I^*$. In other words, $G \subset I$ is a standard basis of $I$, if $I^*$ is generated by $g^*$ for $g \in G $.
\end{definition} 

We use the following crucial fact in order to relate Betti numbers of tangent cones of liftings.
\begin{lemma} (\cite[Lemma 1.2]{HS}) \label{HS} Let $I$ be an ideal of $R = K[x_1,\dots,x_n]$ with $I \subset  \mathfrak{m} = (x_1,\dots,x_n)$. Suppose that $x_1$ is a non zero-divisor on $K[[x_1,\dots,x_n]]/IK[[x_1,\dots,x_n]]$. Let $\pi : R \rightarrow \bar{R}=K[x_2,\dots,x_n]$ be the $K$-algebra homomorphism with $\pi(x_1) = 0$ and $\pi(x_i) = x_i$ for $i > 1$, and set $\bar{I}=\pi(I)$. Let $g_1,\dots,g_r$ be a standard basis of $\bar{I}$ such that there exist polynomials $f_1,\dots,f_r \in I$ with $\pi(f_i) = g_i$ and $\deg (f^*_i) = \deg (g^*_i)$, for $i = 1,\dots,r$. Then,\\
(a) $f_1,\dots,f_r$ is a standard basis of $I$; \\
(b) $x_1$ is regular on $gr_{\m}(R/I)$;\\
(c) there is an isomorphism
$$gr_{\m}(R/I)/x_1 gr_{\m}(R/I) = gr_{\bar{{\m}}}(\bar{R}/\bar{I}),$$
of graded $K$-algebras, where $ \bar{{\m}} = \pi({\m}).\hfill \Box$

\end{lemma}
Before we state our final result, recall that the curve $C$ is of homogeneous type if $\beta_i (R/I_{S})=\beta_i (gr_{\m}(R/I_{S}))$, for all $i$.
\begin{theorem}  \label{resTC}Let $C_k$ be the $k$-th lifting of $C$. Then, there exists a positive integer $k_0$ such that for all $i=1,\dots,n-1$ and $k\geq k_0$,
 $$\beta_i (gr_{\m}(R/I_{S_k}))=\beta_i (gr_{\m}(R/I_{S_{k_0}})).$$ 
Furthermore, when the tangent cone of $C$ is Cohen-Macaulay, $C$ is of homogeneous type if and only if $C_k$ is of homogeneous type for all $k>1$.
\end{theorem}
\begin{proof}
We first notice that $\bar{I}=\pi(I_{S_k})$ is independent of the value of $k$, as $\pi(B_k)=\pi(B)=B$ if $B$ does not involve $x_1$ and $\pi(B_k)=\pi(M-N_k)=M$ if $x_1$ divides $N_k$. Being the image of ideals with binomial generators, $\bar{I}$ have a standard basis consisting of binomials and monomials. These binomials are images of themselves under $\pi$ as they do not involve the variable $x_1$. So, we need to prove that for any monomial $M=x_2^{u_2}\cdots x_n^{u_n}$ in this standard basis, there is a binomial $B_k=M-N_k$ in $I_{S_k}$ with $\pi(B_k)=M$ and $\deg(B^*_k)=\deg(M)$. Since $M\in \bar{I}$,  there is always a binomial $B_k=M-N_k$ in $I_{S_k}$ with $\pi(B_k)=M$ but the last condition is satisfied exactly when $\deg (M) \leq \deg (N_k)$. In the proof of Theorem \ref{liftsofCM}, we demonstrate that there is some positive integer $k_0$ such that for all $k\geq k_0$, $\deg (M) \leq \deg (N_k)$ is satisfied and thus the tangent cone $gr_{\m}(R/I_{S_{k}})$ is Cohen-Macaulay for all $k\geq k_0$. Therefore, hypothesis of Lemma \ref{HS} holds as $x_1$ is always regular on $K[[S_k]]$. Thus, we have the following isomorphism
$$gr_{\m}(R/I_{S_k})/x_1 gr_{\m}(R/I_{S_k}) = gr_{\bar{{\m}}}(\bar{R}/\bar{I}),$$
which implies that $$\beta_i (gr_{\m}(R/I_{S_k})/x_1 gr_{\m}(R/I_{S_k}))=\beta_i (gr_{\bar{{\m}}}(\bar{R}/\bar{I})),\quad \mbox{for all}\quad  i=1,\dots,n-1.$$ 
Since $x_1$ is not a zero-divisor on $gr_{\m}(R/I_{S_k})$, it follows that 
$$\beta_i (gr_{\m}(R/I_{S_k})/x_1 gr_{\m}(R/I_{S_k}))=\beta_i(gr_{\m}(R/I_{S_k})),\quad \mbox{for all}\quad  i=1,\dots,n-1.$$
Therefore, for all $k\geq k_0$ we have the following
 $$\beta_i (gr_{\m}(R/I_{S_k}))=\beta_i (gr_{\bar{{\m}}}(\bar{R}/\bar{I}))=\beta_i (gr_{\m}(R/I_{S_{k_0}})),\quad \mbox{for all}\quad  i=1,\dots,n-1.$$ 
 When the tangent cone $gr_{\m}(R/I_{S})$ of $C$ is Cohen-Macaulay, we have $k_0=1$ by the proof of Theorem \ref{liftsofCM}. Thus, by the first part, we have $$\beta_i (gr_{\m}(R/I_{S_k}))=\beta_i (gr_{\m}(R/I_{S})),\quad \mbox{for all}\quad  i=1,\dots,n-1$$ 
and by Proposition \ref{free}, we have $\beta_i (R/I_{S_k})=\beta_i (R/I_{S})$, for all $  i=1,\dots,n-1$. Therefore, $C$ is of homogeneous type if and only if  $C_k$ is of homogeneous type.
\end{proof}

\section*{Acknowledgements} 
The author thanks the anonymous referee for comments improving the presentation of the paper.

\bibliographystyle{amsplain}

\end{document}